\documentclass[11pt]{amsart}
\usepackage{amssymb}
\usepackage{amsthm}
\usepackage{amscd}
\usepackage{hyperref}

\newtheorem{theorem}{Theorem}[section]
\newtheorem{lemma}[theorem]{Lemma}

\newtheorem{question}[theorem]{Question}
\newtheorem{conjecture}[theorem]{Conjecture}

\theoremstyle{definition}
\newtheorem{definition}[theorem]{Definition}

\makeatletter

\def\dotminussym#1#2{%
  \setbox0=\hbox{$\m@th#1-$}%
  \kern.5\wd0%
  \hbox to 0pt{\hss\hbox{$\m@th#1-$}\hss}%
  \raise.6\ht0\hbox to 0pt{\hss$\m@th#1.$\hss}%
  \kern.5\wd0}

\mathchardef\mhyphen="2D

\allowdisplaybreaks[2]

\newcommand{\RCA}{\ensuremath{\mathbf{RCA_0}}}

\newcommand{\ACA}{\ensuremath{\mathbf{ACA_0}}}
\newcommand{\ATR}{\ensuremath{\mathbf{ATR_0}}}
\newcommand{\Pioo}{\ensuremath{\mathbf{\Pi^1_1\mhyphen CA_0}}}

\newcommand{\RCAU}{\ensuremath{\mathbf{RCA_0+\exists\mathfrak{U}}}}
\newcommand{\ACAU}{\ensuremath{\mathbf{ACA_0+\exists\mathfrak{U}}}}
\newcommand{\ATRU}{\ensuremath{\mathbf{ATR_0+\exists\mathfrak{U}}}}
\newcommand{\PiooU}{\ensuremath{\mathbf{\Pi^1_1\mhyphen CA_0+\exists\mathfrak{U}}}}

\begin{document}

\title{Ultrafilters in Reverse Mathematics}
\author{Henry Towsner}
\date{\today}
\thanks{Partially supported by NSF grant DMS-1157580.}
\address {Department of Mathematics, University of Connecticut U-3009, 196 Auditorium Road, Storrs, CT 06269-3009, USA}
\email{henry.towsner@uconn.edu}
\urladdr{www.math.uconn.edu/~towsner}

\begin{abstract}
We extend theories of reverse mathematics by a non-principal ultrafilter, and show that these are conservative extensions of the usual theories \ACA{}, \ATR{}, and \Pioo{}.
\end{abstract}

\maketitle

\section{Introduction}
A recurring difficulty in reverse mathematics is adapting a proof which involves higher-order notions to take place in the purely second-order context of the major theories of reverse mathematics.  One of the most common such notions is that of a non-principal ultrafilter on $\mathbb{N}$.  Recall that a non-principal ultrafilter is a set $\mathfrak{U}\subseteq\mathcal{P}(\mathbb{N})$ such that:
\begin{itemize}
  \item If $S\in \mathfrak{U}$ and $T\in \mathfrak{U}$ then $S\cap T\in \mathfrak{U}$,
  \item If $S\in\mathfrak{U}$ and $S\subseteq T$ then $T\in\mathfrak{U}$,
  \item Every element of $\mathfrak{U}$ is infinite,
  \item For every $S\subseteq\mathbb{N}$, either $S\in\mathfrak{U}$ or $\mathbb{N}\setminus S\in\mathfrak{U}$.
\end{itemize}
Not only is the statement itself intrinsically third-order, but there are no ``natural'' examples of such objects---their existence cannot be proved even in pure ZF, let alone in the much weaker theories of reverse mathematics.

Ultrafilters have turned out to be useful tools in combinatorics and dynamical systems (see \cite{Hindman2005} for many examples), and there have been several successful translations of proofs that use ultrafilters into proofs that can be carried out in second-order arithmetic \cite{avigad98,hirst04,towsner:MR2791353,towsner:MR2795548}.  These translations all depend on the same idea: in a proof of a second-order statement, a full ultrafilter can be replaced by a filter in which the fourth condition above holds not for all sets $S$, but only for a sufficiently large countable collection of particular sets.  Since countable collections of sets can be coded by a single set, this ``approximate ultrafilter'' can be described, and even constructed explicitly, in second-order arithmetic.  (These approximate ultrafilters are quite natural objects in their own right---they represent closed sets in the Stone-\v {C}ech compactification of $\mathbb{N}$.)

In this paper we apply this idea systematically: we consider an extension of second-order arithmetic by a third-order predicate and axioms stating that this predicate names an ultrafilter, and show that this is actually a conservative extension for the three stronger theories of reverse mathematics, \ACA{}, \ATR{}, and \Pioo{}.

\section{Theories of Reverse Mathematics with Ultraproducts}
The language $\mathcal{L}^2$ is the usual language of second-order arithmetic, containing two sorts, one for natural numbers (usually denoted with lowercase letters) and one for sets (usually denoted with uppercase letters).  (\cite{simpson99} is the standard reference for the theories of second-order arithmetic we will be considering.)  It will be convenient to assume that $\mathcal{L}^2$ includes, for each second-order term $T$ and each first-order term $t$, a second-order term $T_t$, and that all theories include an axiom specifying that $s\in T_t\leftrightarrow (t,s)\in T$ (where $(s,t)$ abbreviates some quantifier-free formula for pairing).

Since we cannot discuss ultrafilters directly in second-order arithmetic, we will extend the language by adding a unary predicate on second-order terms which will be intended to denote a non-principal ultrafilter.
\begin{definition}
We define the language $\mathcal{L}^{\mathfrak{U}}$ to be the language $\mathcal{L}^{2}$ together with a unary predicate $\mathfrak{U}$ on second-order terms.
\end{definition}

The main axioms defining the properties of $\mathfrak{U}$ are given by:
\begin{definition}
We define the collection of axioms $\exists\mathfrak{U}$ to consist of the axioms:
\begin{enumerate}
  \item $\forall^2 X(X\in\mathfrak{U}\rightarrow \forall x\exists y>x y\in X)$,
  \item $\forall^2 X,y(X\in\mathfrak{U}\wedge Y\in\mathfrak{U}\rightarrow X\cap Y\in\mathfrak{U})$,
  \item $\forall^2 X,y(X\in\mathfrak{U}\wedge X\subseteq Y\rightarrow Y\in\mathfrak{U})$, and
  \item $\forall^2 X(X\in\mathfrak{U}\vee X^c\in\mathfrak{U})$.
\end{enumerate}
\end{definition}
We need the operation $X_t$, which is usually treated as a defined operation, to get around a technical syntactic limitation.  When reasoning informally in \ACA{}, it is normal to take two arithmetic formulas $\phi(x,X)$ and $\psi(x,y)$ and form the set
\[\{n\mid \phi(n,\{m\mid \psi(n,m)\})\}.\]
Strictly speaking, however, this is not valid in the language $\mathcal{L}^2$: we have to replace the $X$ in $\phi$ with the formula $\psi$.  In the presence of $\mathfrak{U}$, this is no longer possible, since in the formula $X\in\mathfrak{U}$ we cannot the second-order predicate $X$ with a formula $\psi$.

With the terms $X_n$, however, it is possible to deal with this: the set can be written
\[\{n\mid \phi(n,X_n)\}\]
where $X_n$ is a free variable, and then we can carry our our proof under the assumption that $X=\{(n,m)\mid \psi(n,m)\}$.

\begin{definition}
If $T$ is a theory of second-order arithmetic, we define $T+\exists\mathfrak{U}$ to be the theory in $\mathcal{L}^{\mathfrak{U}}$ whose axioms consist of the axioms of $T$ together with the axioms of $\exists\mathfrak{U}$. 
\end{definition}

\begin{definition}
We write $Ult(U)$ for the formula of $\mathcal{L}^2$ stating that for every finite set $F$, $\bigcap_{n\in F}U_n$ is infinite.
\end{definition}

\begin{theorem}
The theory \RCAU{} implies \ACAU{}.
\end{theorem}
\begin{proof}
It suffices to show that \RCAU{} proves Ramsey's Theorem for triples, since this is known to imply arithmetic comprehension (see \cite{simpson99}).

First, observe that we can prove that whenever we have a partition $\mathbb{N}=S_0\cup\cdots\cup S_a$, there is a (necessarily unique) $b\leq a$ such that $S_b\in\mathfrak{U}$.  (This is trivial when $a$ is a genuine natural number, but in general $a$ might be some nonstandard number.)  This is because we may apply induction to find a least $b$ such that $\bigcup_{c\leq b}S_c\in\mathfrak{U}$, and since $\bigcup_{c\leq b-1}S_c\not\in\mathfrak{U}$, we must have $S_b\in\mathfrak{U}$.

Next we show a strong form of Ramsey's Theorem for pairs.  Suppose that for each $n$, we have $c_n:[\mathbb{N}]^2\rightarrow\{0,1\}$; we claim there is an infinite sequence $h_0<h_1<\cdots$ so that for every $n$, $c_n$ is monochromatic on $[\{h_n,h_{n+1},\ldots\}]^2$.  For each $x$, we may induce a coloring on $\mathbb{N}\setminus[0,x]$ by $c^x(y)=\{n\leq x\mid c_n(x,y)=0\}$.  This is a finite coloring, so there is some $S_x\subseteq [0,x]$ such that $\{y\mid c^x(y)=S\}\in\mathfrak{U}$.  For each $n$, let $T_n=\{x\mid n\in S_x\}$.  We now inductively construct a sequence $h_0<h_1<\cdots$ so that if $n\leq i$ then $n\in S_{h_i}$ iff $T_n\in\mathfrak{U}$, and if $i<j$ then $c^{h_i}(h_j)=S_{h_i}$.  Given $h_0,\ldots,h_{n}$, this requires that $h_{n+1}$ be a member of a finite list of sets, all belonging to $\mathfrak{U}$, so the intersection of these sets is infinite and we may choose $h_{n+1}$ to be the least element of the intersection greater than $h_n$.  Given $n\leq i<j$, we have $c^{h_i}(h_j)=S_{h_i}$ and $n\in S_{h_i}$ iff $T_n\in\mathfrak{U}$, so $c_n(h_i,h_j)=0$ iff $T_n\in\mathfrak{U}$; in particular, this is independent of the choice of $i,j$.

Finally, to show Ramsey's Theorem for triples, let $c:[\mathbb{N}]^3\rightarrow\{0,1\}$ be given.  For each $n$, we induce a coloring $c_n(x,y)=c(n,x,y)$, and so we may choose an infinite sequence $h_0<h_1<\cdots$ so that whenever $h_i<j<k$, $c(h_i,h_j,h_k)$ depends only on $h_i$.  The color induced by $i$ is computable, so by the infinite pigeonhole principle, we may restrict to a subset where $c(h_i,h_j,h_k)$ is constant so long as $h_i<j<k$.  Finally, by taking the subsequence $k_0=h_0$, $k_{n+1}=h_{k_n+1}$, we obtain a subsequence where $c$ is monochromatic.
\end{proof}

\section{Forcing}
\begin{definition}
If $Ult(U)$ holds, we say $U$ is a \emph{condition}.  We write $V\preceq U$ if $Ult(V)$ and for each $n$ there is an $m$ such that $V_m\subseteq U_n$.

We define by recursion on a formula $\phi$ of $\mathcal{L}^{\mathfrak{U}}$ with free variables $\vec x,\vec X$ a formula $U\Vdash\phi$ in $\mathcal{L}^2$ with free variables $\vec x,\vec X,U$ by:
\begin{enumerate}
  \item If $\phi$ is an atomic formula not containing $\mathfrak{U}$, $U\Vdash \phi$ is simply $\phi$,
  \item If $\phi$ is $T\in\mathfrak{U}$ then $U\Vdash \phi$ is the formula ``there is a finite $F$ such that $\bigcap_{n\in F} U_n\setminus T$ is finite''
  \item $U\Vdash\phi\wedge\psi$ is the formula $(U\Vdash\phi)\wedge(U\Vdash\psi)$,
  \item $U\Vdash\phi\vee\psi$ is the formula $\forall^2 V\preceq U\exists^2 W\preceq V (W\Vdash\phi\vee W\Vdash\psi)$,
  \item $U\Vdash\neg\phi$ is the formula $\forall^2 V\preceq U V\not\Vdash\phi$,
  \item $U\Vdash\phi\rightarrow\psi$ is the formula $\forall^2 V\preceq U (V\Vdash \phi\rightarrow V\Vdash \psi)$,
  \item $U\Vdash\forall x\phi$ is the formula $\forall^2 x U\Vdash\phi$,
  \item $U\Vdash\exists x\phi$ is the formula $\forall^2 V\preceq U \exists^2 W\preceq V \exists x W\Vdash\phi$,
  \item $U\Vdash\forall^2 X\phi$ is the formula $\forall^2 X U\Vdash\phi$,
  \item $U\Vdash\exists^2 X\phi$ is the formula $\forall^2 V\preceq U \exists^2 W\preceq V \exists^2 X W\Vdash\phi$,
\end{enumerate}

We write $\Vdash\phi$ for $\emptyset\Vdash\phi$.
\end{definition}

We may take $\neg\phi$ to be an abbreviation for $\phi\rightarrow 0\neq 1$, and then take $\vee,\exists,\exists^2$ to be abbreviations for their de Morgan equivalents.  It is easy to check that the definition of forcing remains unchanged, and this allows us to consider only the atomic, $\wedge$, $\rightarrow$, $\forall$, and $\forall^2$ cases when we give inductive proofs.

\begin{lemma}
For each formula $\phi$ with free variables $\vec x,\vec X$, \RCA{} proves
\[\forall^2 U\forall^2 \vec X\forall \vec x\forall^2 V\preceq U (U\Vdash\phi\rightarrow V\Vdash\phi).\]
\label{monotonicity}
\end{lemma}
\begin{proof}
By induction on $\phi$.  (Since this is the first of many similar arguments, we point out explicitly that the induction on $\phi$ is being carried out externally to the theory \RCA{}.  That is, in this and later proofs we are working in ordinary mathematics reasoning about inductively about the formal theory \RCA{}.)

Let $U,\vec X,\vec x,V\preceq U$ be given.
\begin{enumerate}
  \item If $\phi$ is an atomic formula not containing $\mathfrak{U}$, this is immediate from the definition,
  \item If $\phi$ is $T\in\mathfrak{U}$ then $U\Vdash\phi$ means there is a finite $F$ such that $\bigcap_{n\in F}U_n\setminus T$ is finite; taking $G$ to be a set so that for each $n\in F$ there is an $m\in G$ such that $V_m\subseteq U_n$, we have $\bigcap_{m\in G}V_m\subseteq\bigcap_{n\in F}U_n$, so $V\Vdash\phi$,
  \item If $\phi$ is $\psi\wedge\chi$, the claim follows immediately from IH,
  \item If $\phi$ is $\forall y\psi$ then by IH we have $\forall y(U\Vdash\psi\rightarrow V\Vdash\psi)$, and therefore $U\Vdash\forall y\psi\rightarrow V\Vdash\forall y\psi$,
  \item The case for $\forall^2 Y\psi$ is similar,
  \item If $\phi$ is $\psi\rightarrow\chi$ and $U\Vdash\psi\rightarrow\chi$ then whenever $W\preceq V$ and $W\Vdash\psi$, also $W\preceq U$, and therefore $W\Vdash\chi$, so $V\Vdash\psi\rightarrow\chi$.
\end{enumerate}
\end{proof}

\begin{theorem}
For each formula $\phi$ of $\mathcal{L}^2$ with free variables $\vec x,\vec X$, \RCA{} proves
\[\forall^2 U \forall^2 \vec X\forall \vec x((U\Vdash\phi)\leftrightarrow\phi).\]
\label{reflection}
\end{theorem}
\begin{proof}
By induction on $\phi$.  Let $U,\vec X,\vec x$ be given.
\begin{enumerate}
  \item If $\phi$ is an atomic formula not containing $\mathfrak{U}$, this is immediate from the definition,
  \item Atomic formulas containing $\mathfrak{U}$ may not appear in $\phi$,
  \item If $\phi$ is $\psi\wedge\chi$ then by IH we have $(U\Vdash\psi)\wedge(U\Vdash\chi)\leftrightarrow(\psi\wedge\chi)$, and the left-hand side is the definition of $U\Vdash\psi\wedge\chi$.
  \item If $\phi$ is $\forall y\psi$ then by IH we have $\forall y(U\Vdash\psi\leftrightarrow\psi)$, and so we have $(\forall y U\Vdash\psi)\leftrightarrow\forall y\psi$, as desired,
  \item The case for $\forall^2 Y\psi$ is similar,
  \item If $\phi$ is $\psi\rightarrow\chi$, we prove the two directions separately.  Suppose $U\Vdash\psi\rightarrow\chi$; if $\psi$ holds then by IH, $U\Vdash\psi$, and therefore $U\Vdash\chi$, so again by IH, $\chi$ holds, and therefore $\psi\rightarrow\chi$ holds, while if $\psi$ fails then $\psi\rightarrow\chi$ holds trivially.  If $\psi\rightarrow\chi$ holds then either $\psi$ fails or $\chi$ holds; in the latter case, $U\Vdash\chi$ by IH, so $U\Vdash\psi\rightarrow\chi$.  In the former case, since $\psi$ fails, for all $V$ we have $V\not\Vdash\psi$, and therefore $U\Vdash\psi\rightarrow\chi$.
\end{enumerate}
\end{proof}

In particular, this means that if $\sigma$ is a sentence of $\mathcal{L}^2$ and $T$ is any theory extending \RCA{} which proves that $U\Vdash\sigma$ then also $T\vdash\sigma$.  The next step will be showing that whenever $\ACAU\vdash\phi$, \ACA{} proves $\Vdash\phi$.  Naturally, we will show this by induction on proofs, demonstrating that \ACA{} proves that it can force all axioms and rules of \ACAU{}.

We start with the logical axioms; the double negation law will require a bit of work.

\begin{lemma}
For each $\phi$ with free variables $\vec x,\vec X$, \RCA{} proves $\Vdash\forall^2\vec X\forall x(\phi\rightarrow\neg\neg\phi)$.
\end{lemma}
\begin{proof}
Note that $U\Vdash\neg\neg\phi$ is equivalent to $\forall V\preceq U \exists W\preceq V W\Vdash\phi$.  Let $\vec X,\vec x$ be given, and suppose $U\Vdash\phi$.  We must show $U\Vdash\neg\neg\phi$, so let $V\preceq U$ be given.  Since $V\Vdash\phi$ by Lemma \ref{monotonicity},  it follows that there is some $W\preceq V$ such that $W\Vdash\phi$.  This holds for any $V\preceq U$, so $U\Vdash\neg\neg\phi$.
\end{proof}

\begin{lemma}
For each $\phi$ with free variables $\vec x,\vec X$, \RCA{} proves $\Vdash\forall^2\vec X\forall\vec x(\neg\neg\phi\rightarrow\phi)$.
\end{lemma}
\begin{proof}
We proceed by induction on $\phi$.
\begin{enumerate}
  \item Suppose $\phi$ is atomic and does not contain $\mathfrak{U}$.  If $U\Vdash\neg\neg\phi$ then there is some $W\preceq U$ such that $W\Vdash\phi$, and therefore $\phi$ must be true, so $U\Vdash\phi$,
  \item If $\phi$ is $T\in\mathfrak{U}$ but $U\not\Vdash\phi$ then there is no finite $F\subseteq U$ such that $\bigcap F\setminus T$ is finite.  It follows that $Ult(U\cup\{T^c\})$ holds, and therefore $U\cup\{T^c\}\preceq U$ and $U\cup\{T^c\}\Vdash\neg\phi$, so $U\not\Vdash\neg\neg\phi$,
  \item If $\phi$ is $\forall y\psi$ and $U\Vdash\neg\neg\phi$, we have $\forall V\preceq U\exists W\preceq V\forall y W\Vdash\phi$, and so also $\forall y(\forall V\preceq U\exists W\preceq V W\Vdash\phi)$, so $\forall y U\Vdash\neg\neg\psi$, and so by IH, $\forall y U\Vdash\psi$.  The $\forall^2$ and $\wedge$ cases are similar,
  \item If $\phi$ is $\psi\rightarrow\chi$ and $U\Vdash\neg\neg\phi$, consider any $V\preceq U$ with $V\Vdash\psi$.  We will show $V\Vdash\neg\neg\chi$, which by IH implies $V\Vdash\chi$.  Let $W\preceq V$ be given; then since $W\preceq U$, there is a $W'\preceq W$ such that $W'\Vdash\psi\rightarrow\chi$.  Since $W'\preceq V$, so $W'\Vdash\psi$, we have $W'\Vdash\chi$.  This shows that $V\Vdash\neg\neg\chi$, as desired.
\end{enumerate}
\end{proof}

\begin{lemma}
If $\phi$ is a logical axiom with free variables $\vec x,\vec X$ then \RCA{} proves $\Vdash\forall^2 \vec X\forall \vec x\phi$.
\label{log_axiom}
\end{lemma}
\begin{proof}
It is easy to check that most logical axioms in one's preferred system are easily satisfied by checking the definitions.  The double negation law is covered by the previous lemma.

The only other axioms we explicitly check are the quantifier axioms.  Observe that $\Vdash(\forall x \phi\rightarrow\phi[t/x])$ since if $U\Vdash\forall x\phi$, we have $\forall x (U\Vdash\phi)$, and therefore $U\Vdash\phi[t/x]$.  The case for $\forall^2 X\phi\rightarrow\phi[Y/X]$ is similar.  To see $\Vdash(\phi[t/x]\rightarrow\exists x\phi)$, observe that if $U\Vdash\phi[t/x]$ then for any $V\preceq U$, $V\Vdash\phi[t/x]$, so $\exists x V\Vdash\phi$, so $U\Vdash\exists x\phi$.  Again, the case for $\exists^2$ is similar.
\end{proof}

 \begin{lemma}
 For any $\phi,\psi$, \RCA{} proves that $\forall^2 U\forall^2\vec X\forall\vec x(U\Vdash\phi\wedge U\Vdash(\phi\rightarrow\psi)\rightarrow U\Vdash\psi)$.
 \label{rule_first}
 \end{lemma}
 \begin{proof}
 Suppose $U\Vdash\phi$ and $U\Vdash(\phi\rightarrow\psi)$.  Then for any $V\preceq U$ such that $V\Vdash\phi$, $V\Vdash\psi$.  But $U$ itself is such a $V$, so $U\Vdash\psi$.
 \end{proof}

Next we turn to the axioms of second order arithmetic.
\begin{definition}
We say $U$ \emph{decides} $\phi$ (with respect to values $\vec x,\vec X$ for the free variables in $\phi$) if either $U\Vdash\phi$ or $U\Vdash\neg\phi$.

Let $\phi(\vec x,\vec X)$ be an arithmetic formula with only the displayed free variables; we say $U$ \emph{recursively decides} $\phi$ with respect to $\vec X$ if:
\begin{enumerate}
  \item $\phi$ is atomic and does not contain $\mathfrak{U}$,
  \item $\phi$ is $T\in\mathfrak{U}$ and for every $\vec x$ there is a finite $F$ such that either $\bigcap_{n\in F}U_n\setminus T$ is finite or $\bigcap_{n\in F}U_n\cap T$ is finite,
  \item $\phi$ is $\neg\psi$ and $U$ recursively decides $\psi$ with respect to $\vec X$,
  \item $\phi$ is $\psi\wedge\chi$, $\psi\vee\chi$, or $\psi\rightarrow\chi$ and $U$ recursively decides $\psi$ and $\chi$ with respect to $\vec X$,
  \item $\phi$ is $\forall x\psi$ or $\exists x\psi$ and $U$ recursively decides $\psi$ with respect to $\vec X$.
\end{enumerate}
\end{definition}

It is easy to see by recursion that:
\begin{lemma}
If $U$ recursively decides $\phi$ then for each $\vec x$ $U$ decides $\phi$ and $\{\vec x\mid U\Vdash\phi\}$ is arithmetic (in the parameters $\vec X$).
\end{lemma}

\begin{lemma}
\ACA{} proves that for any condition $U$ and any $X$, there is a $V\preceq U$ such that for every $i$, $U$ decides $X_i\in\mathfrak{U}$.
\label{settle_term}
\end{lemma}
\begin{proof}
Given a $\{0,1\}$-sequence $\sigma$ and $i<|\sigma|$, write $X_{i}(\sigma)$ for $X_{i}$ if $\sigma(i)=1$ or $X_{\vec x}^c$ if $\sigma(i)=0$.  Consider those sequences $\sigma$ such that for each $i$ with $i<|\sigma|$, $\sigma(i)=1$ iff for all finite sets $F$, $\bigcap_{n\in F}U_n\cap\bigcap_{j<i}X_{j}(\sigma)\cap X_{i}$ is infinite.  By induction $i$, we may observe that for every finite set $F$, $\bigcap_{n\in F}U_n\cap\bigcap_{j\leq i)}X_i(\sigma)$ is infinite: when $\sigma(i)=1$ this follows from the definition, and when $\sigma(i)=0$ we must have $\bigcap_{n\in F}U_n\cap\bigcap_{j<i}X_{j}(\sigma)$ infinite (by IH), and therefore since $\bigcap_{n\in F}U_n\cap\bigcap_{j<i}X_{j}(\sigma)\cap X_{i}$ is finite, $\bigcap_{n\in F}U_n\cap\bigcap_{j<i}X_{j}(\sigma)\setminus X_{i}=\bigcap_{n\in F}U_n\cap\bigcap_{j<i}X_{j}(\sigma)\cap X^c_{i}$ is infinite.  Each such sequence has a unique extension, so we may easily construct the set of such sequences, then the unique $\Lambda$ extending all such sequences, and then the set $V$ is given by $V_{2n}=U_n$, $V_{2 i+1}=X_{i}(\Lambda)$.  By construction we have $Ult(V)$ and $V\preceq U$, and clearly either $V\Vdash X_i\in\mathfrak{U}$ or $V\Vdash X_i^c\in\mathfrak{U}$ for all $i$.
\end{proof}

\begin{lemma}
Let $\phi(\vec x,\vec X)$ be an arithmetic formula with only the displayed free variables.  Then \ACA{} proves that for each $\vec X$ and any $U$, there is a $V\preceq U$ such that $V$ recursively decides $\phi$ with respect to $\vec X$.
\label{arith_rec_decide}
\end{lemma}
\begin{proof}
We show this by induction on $\phi$.  If $\phi$ is atomic and does not contain $\mathfrak{U}$ then $U$ suffices.  If $\phi$ has the form $T\in\mathfrak{U}$ then this is Lemma \ref{settle_term}.  The cases for $\wedge$, $\rightarrow$, and $\forall$ follow immediately from IH.
\end{proof}

\begin{lemma}
If $\phi$ is an axiom of \ACAU{} then \ACA{} proves $\Vdash\forall^2\vec X\forall \vec x\phi$.
\label{arith_axiom}
\end{lemma}
\begin{proof}
The basic axioms and induction axiom never involve $\mathfrak{U}$, so are immediately forced by Theorem \ref{reflection}.

We turn to the comprehension axiom,
\[\exists Y\forall n(n\in Y\leftrightarrow\phi(n,\vec x,\vec X))\]
where $\phi$ is arithmetic.  Fix $\vec x,\vec X, U$.  By the previous lemma, choose $V\preceq U$ such that $V$ recursively decides $\phi$.  Then $Y=\{n\mid V\Vdash\phi(n,\vec x,\vec X)\}$ is arithmetic, so $V\Vdash \forall n(n\in Y\leftrightarrow\phi(n,\vec x,\vec X))$, so in particular $\exists Y(V\Vdash \forall n(n\in Y\leftrightarrow\phi(n,\vec x,\vec X)))$.

Finally, we deal with the axioms in $\exists\mathfrak{U}$.  The first three follow immediately from the definition of $U\Vdash T\in\mathfrak{U}$ and the fact that any $U$ satisfies $Ult(U)$.  For the final axiom, to show that $\Vdash X\in\mathfrak{U}\vee X^c\in\mathfrak{U}$, let $U$ be given.  If $U\Vdash X\in\mathfrak{U}$ then we are done.  Otherwise, whenever $F$ is finite, $\bigcap_{n\in F}U_n\setminus X$ is infinite.  It follows that $Ult(U\cup\{X^c\})$ holds, and therefore $U\cup\{X^c\}\preceq U$ satisfies $U\cup\{X^c\}\Vdash X^c\in\mathfrak{U}$.
\end{proof}

\begin{theorem}
\begin{enumerate}
  \item Suppose $\ACAU\vdash\phi$ where $\vec X,\vec x$ are the free variables in $\phi$.  Then \ACA{} proves that $\Vdash\forall^2 \vec X\forall x\phi$.
  \item \ACAU{} is a conservative extension of \ACA{}.
\end{enumerate}
\end{theorem}
\begin{proof}
\begin{enumerate}
  \item By induction on the proof of $\phi$.  If $\phi$ is a logical axiom or an axiom of \ACAU{}, these are covered by Lemmata \ref{log_axiom} and \ref{arith_axiom} respectively.  If $\phi$ is derived by modus ponens, this is covered by IH and Lemma \ref{rule_first}.
  \item If $\ACAU{}\Vdash\phi$ where $\phi$ is a sentence of $\mathcal{L}^2$, we have that \ACA{} proves $\Vdash\phi$ by the previous part, and therefore \ACA{} proves $\phi$ by Theorem \ref{reflection}.
\end{enumerate}
\end{proof}

We have similar results for \ATR{} and \Pioo{}.

\begin{lemma}
If $\phi$ is an axiom of \ATRU{} then \ATR{} proves $\Vdash\forall^2\vec X\forall \vec x\phi$.
\end{lemma}
\begin{proof}
The only axiom not covered above is the transfinite recursion axiom.  Let $U$ be given so that $U\Vdash WO(\prec_X)$; we must construct a $V\preceq U$ and a $Y$ so that for each $a\in field(\prec_X)$, $Y_a=\theta(Y^a)$.  (Recall that $Y^a=\{(b,n)\mid b\prec_X a\wedge n\in Y_b$.)  Set $V_{0}=U$ and $Y_{-1}=Y^0=\emptyset$.  Given $V_b$ for all $b\prec_X a$, observe that we may define, arithmetically (in $\prec_X$) a set $V^a$ so that $V^a\preceq V_b$ by setting $V_{(b,n)}=(V_b)_n$ for each $b\prec_X a$.  Given $V^a,Y^a$, we may define, arithmetically, a $V_a\preceq V^a$ so that $V_a$ recursively decides $\theta$ with respect to $Y^a$.  Then we may take $Y_a=\{n\mid V_a\Vdash\theta(n,Y^a)\}$, and $Y_a$ is arithmetic in $V_a,Y^a$, and so also in $V^a,Y^a$.  Therefore by trasfinite recursion along $\prec_X$, we may construct a sequence $V_a,Y_a$ so that if $b\prec_X a$ then $V_b\prec_X V_a$ and $V_a\Vdash Y_a=\theta(Y^a)$.  If we set $V_{(a,n)}=(V_a)_n$ for each $a\in field(\prec X)$, we have that $V\Vdash \forall a\in field(\prec_X) Y_a=\theta(Y^a)$.  Since $V\preceq U$, this witnesses that we have forced the transfinite recursion axiom.
\end{proof}

\begin{lemma}
Let $\phi(\vec x,\vec X)$ be a $\Sigma^1_1$ formula with only the displayed free variables.  Then \Pioo{} proves that for any values $\vec x,\vec X$ and any $U$, there is a $V\preceq U$ such that for every $\vec x$, $V$ decides $\phi(\vec x,\vec X)$ and the set of $\vec x$ such that $V\Vdash\phi(\vec x,\vec X)$ can be expressed by a formula arithmetic in $\Pi^1_1$ formulas.
\end{lemma}
\begin{proof}
We have $\phi(\vec x,\vec X)=\exists Y\psi(\vec x,Y,\vec X)$.  Let $U,\vec X$ be given.  Consider those $V\preceq U$ and $Y$ such that for every $\vec x$, $V$ recursively decides $\psi(\vec x,Y_{\vec x},\vec X)$.  By Lemma \ref{arith_rec_decide}, the following is an arithmetic formula on $U,V,W,Y,\vec X$:
\begin{itemize}
  \item $V\preceq U$,
  \item $V$ recursively decides $\psi(\vec x,Y_{\vec x},\vec X)$,
  \item $W=\{\vec x\mid V\Vdash\psi(\vec x,Y_{\vec x},\vec X)\}$,
\end{itemize}
In particular, this is equivalent to a formula of the form $\exists V,Y,W,K\forall y\chi$ where $\chi$ contains only bounded quantifiers and $K$ is some set coding witnesses to the original statement.  There is some tuple satisfying this, namely any extension of $U$ recursively deciding $\psi(m,\emptyset)$ for all $m$, together with $Y=\emptyset$.

Consider the tree of finite tuples of sequences $(\sigma_V,\sigma_Y,\sigma_W,\sigma_K)$ approximating sets satisfying this formula; we may order such sequences by declaring $(\sigma_V,\sigma_Y,\sigma_W,\sigma_K)\prec(\tau_V,\tau_Y,\tau_W,\tau_K)$ if there is an $n$ such that $\sigma_W\upharpoonright n=\tau_W\upharpoonright n$, $\sigma_W(n)=1$ while $\sigma_W(n)=0$.  Since there is an infinite path through this tree, there is a leftmost path through this tree with respect to $\prec$.  We claim that this path gives the desired $V,Y,W$.  By construction, $V$ recursively decides $\psi(\vec x,Y_{\vec x},\vec X)$ and $W=\{\vec x\mid V\Vdash\psi(\vec x,Y_{\vec x},\vec X)\}$, so it suffices to show that for any $\vec n$, if $V\not\Vdash\forall^2 Y\psi(\vec n,Y,\vec X)$ then $V\Vdash\psi(\vec n,Y,\vec X)$.  But suppose $V\not\Vdash\forall^2 Y\psi(\vec n,Y,\vec X)$; then there is a $V'\preceq V$ and a $Y'$ such that $W\cup\{\vec n\}\subseteq \{\vec x\mid V'\Vdash\psi(\vec x,Y'_{\vec x},\vec X)\}$.  But if $\vec n\not\in W$ then this gives a path through the tree further to the left, contradicting the construction.  Therefore whenever $V\not\Vdash\psi(\vec x,Y_{\vec x})$, we may conclude $V\Vdash\forall^2 Y\psi(\vec x,Y,\vec X)$, so for every $\vec x$, $V$ decides $\exists^2 Y\psi(\vec x,Y,\vec X)$, and further, $W$, given by a formula arithmetic in $\Pi^1_1$ formulas, is precisely the set of parameters such that $V\Vdash\exists^2 Y\psi(\vec x,Y,\vec X)$.
\end{proof}

\begin{lemma}
If $\phi$ is an axiom of \PiooU{} then \Pioo{} proves $\Vdash\forall^2\vec X\forall \vec x\phi$.
\end{lemma}
\begin{proof}
The only new axiom is $\Pi^1_1$-comprehension, which is equivalent to $\Sigma_1^1$-comprehension.  Given $U$, we wish to find a $V\preceq U$ and a $W$ such that
\[V\Vdash\forall n(m\in W\leftrightarrow \phi(n))\]
where $\phi$ is $\Sigma^1_1$.  The previous lemma immediately gives such sets.
\end{proof}

\begin{theorem}
\begin{enumerate}
  \item \ATRU{} is a conservative extension of \ATR{}, and
  \item \PiooU{} is a conservative extension of \Pioo.
\end{enumerate}
\end{theorem}

\section{Adding Better Ultrafilters}
It is natural to ask whether we can strengthen the scheme $\exists\mathfrak{U}$ by requiring that the ultrafilter $\mathfrak{U}$ belong to one of the various classes of ultrafilters that have been studied.

The simplest family of properties we could ask for is to demand that each element of $\mathfrak{U}$ be large in some sense stronger than merely being infinite.  More precisely, we define:
\begin{definition}
A property $\mathfrak{P}\subseteq\mathcal{P}(\mathbb{N})$ is \emph{divisible} \cite{ellis69,glasner80} if:
\begin{enumerate}
  \item $\mathbb{N}\in\mathfrak{P}$,
  \item $\emptyset\not\in\mathfrak{P}$,
  \item Whenever $S\in\mathfrak{P}$ and $S\subseteq T$, $T\in\mathfrak{P}$,
  \item If $S\in\mathfrak{P}$ and $S=S_0\cup S_1$ then either $S_0\in\mathfrak{P}$ or $S_1\in\mathfrak{P}$.
\end{enumerate}
\end{definition}
For some examples of such properties, consider the sets $S$ such that:
\begin{itemize}
  \item $\sum_{n\in S}\frac{1}{n}=\infty$,
  \item $\limsup_{m-n\rightarrow\infty}\frac{|S\cap[n,m]|}{m-n}>0$ (the sets of \emph{positive upper Banach density}),
  \item Those sets such that for some $b$ and every $n$, there is an $x$ so that for each $i\in [x,x+n]$, $[i,i+b]\cap S\neq\emptyset$ (the \emph{piecewise syndetic} sets).
\end{itemize}

When, as in these three cases, $\mathfrak{P}$ is expressed by an arithmetic formula, the proof above immediately generalizes:
\begin{theorem}
If $X\in\mathfrak{P}$ is an arithmetic formula, $T$ is one of \ACA{}, \ATR{}, or \Pioo{}, and $T$ proves that $\mathfrak{P}$ is divisible and that every element of $\mathfrak{P}$ is infinite then $T+\exists\mathfrak{U}+\mathfrak{U}\subseteq\mathfrak{P}$ is a conservative extension of $T$.
\end{theorem}
We need that the property $X\in\mathfrak{P}$ be arithmetic in order to force the arithmetic comprehension axiom, where we have to be able to accumulate those values of the numeric parameters where certain sets belong to $\mathfrak{P}$.  Note that all three of the examples given above satisfy the assumption of this theorem.

Another important example of a divisible property is the \emph{IP sets}: a set $S$ is IP if it contains an infinite set $s_1<s_2<\cdots$ and all sums of finitely many elements of this sequence.  The divisibility of the IP sets is better known as Hindman's Theorem; it is know \cite{blass87} that Hindman's Theorem is provable in $\ACA^+$ and that Hindman's Theorem implies \ACA{} over \RCA{}, but the exact strength is a well studied problem.  The property of being IP is known to be $\Sigma^1_1$-complete, so the result above does not apply.
\begin{conjecture}
\ATRU$+$``every element of $\mathfrak{U}$ is IP'' is a conservative extension of \ATR{}.
\end{conjecture}

The proof would likely involve the use of the Iterated Hindman's Theorem, which is also provable in $\ACA^+{}$ \cite{hirst04}.  Given the difficulty in separating these theorems from \ACA{}, it would be intersting to know whether the even stronger property that one has an entire ultrafilter of IP sets is enough to break out of \ACA{}.
\begin{question}
Does \ACAU$+$``every element of $\mathfrak{U}$ is IP'' imply $\ACA^+$?
\end{question}

A second class of family of properties are those related to the topology and algebra of the space of ultrafilters on $\mathbb{N}$ (see \cite{hindman98}).  One important example is the class of \emph{idempotent ultrafilters}:
\begin{definition}
Given $S$, let $S-n=\{m\mid m+n\in S\}$.  $\mathfrak{U}$ is \emph{idempotent} if whenever $S\in\mathfrak{U}$, $\{n\mid S-n\in\mathfrak{U}\}\in\mathfrak{U}$'.
\end{definition}
It is not hard to see that every element of an idempotent ultrafilted is an IP set, but an ultrafilter consisting only of IP sets can still fail to be idempotent.

\begin{question}
For which $T\in\{\ACA,\ATR,\Pioo\}$ is $T+$``$\mathfrak{U}$ is an idempotent ultrafilter'' conservative over $T$?
\end{question}

A third family of properties are the family of set theoretic properties an ultrafilter might have which are known to be independent of ZFC.  Two of the most important examples of such properties are:
\begin{definition}
$\mathfrak{U}$ is a \emph{$P$-point} if for every partition $\mathbb{N}=S_0\cup S_1\cup\cdots\cup S_n\cup\cdots$ such that for every $n$, $S_n\not\in\mathfrak{U}$, there is a $A\in\mathfrak{U}$ such that for every $n$, $|A\cap S_n|$ is finite.

$\mathfrak{U}$ is \emph{Ramsey} if for every partition $\mathbb{N}=S_0\cup S_1\cup\cdots\cup S_n\cup\cdots$ such that for every $n$, $S_n\not\in\mathfrak{U}$, there is a $A\in\mathfrak{U}$ such that for every $n$, $|A\cap S_n|=1$.
\end{definition}
(\cite{blass:MR2768685} contains a survey of results about these and other properties.)

The existence of such ultrafilters is independent of ZFC, but their existence easily follows from the continuum hypothesis, in both cases because it is possible (in ZFC) to extend a countable filter to satisfy a single instance of the property.  Since the forcing construction in the previous section uses precisely such a construction, we suspect the proof can be adapted to these properties.
\begin{conjecture}
For $T\in\{\ACA,\ATR,\Pioo\}$, $T+$``$\mathfrak{U}$ is a $P$-point'' and $T+$``$\mathfrak{U}$ is Ramsey'' are conservative over $T$.
\end{conjecture}
It would be particularly interesting if these properties---which are independent of ZFC---are nonetheless conservative, while properties like being idempotent turn out not to be conservative over these theories despite being provable from ZFC.

\bibliographystyle{plain}
\bibliography{../Bibliographies/main}

\end{document}